\documentclass{amsart}
\usepackage{amsbsy}
\usepackage{amscd}
\usepackage{amsfonts}
\usepackage{amsmath}
\usepackage{amssymb}
\usepackage{amstext}
\usepackage{amsthm}
\usepackage{epsfig}
\usepackage{fancybox}
\usepackage{graphicx}
\usepackage[latin1]{inputenc}
\usepackage{latexsym}

\newtheorem{theorem}{Theorem}[section]
\newtheorem{lemma}[theorem]{Lemma}
\newtheorem{proposition}[theorem]{Proposition}
\newtheorem{corollary}[theorem]{Corollary}
\theoremstyle{definition}
\newtheorem{definition}[theorem]{Definition}

\theoremstyle{remark}

\numberwithin{equation}{section}

\begin{document}
\title{On Young Systems}
\author{Rafael A. Castrequini}
\address{Instituto de Matem\'atica, Estat\'istica e Computa\c{c}\~ao Cient\'ifica,
Universidade Estadual de Campinas, Campinas - SP, Brasil}
\email{rafael.castreq@gmail.com ; pedrojc@ime.unicamp.br}

\author{Pedro J. Catuogno}
\subjclass[2010]{93C30; 93B05;
14F35.}
\keywords{Young integration, $p\,$-variation, Ordinary differential equations, Partial differential equations.}

\begin{abstract}
In this article, we study differential equations driven by continuous paths with with bounded $p$-variation for $1 \leq p< 2$ (Young systems). The most 
important class of examples of theses equations is given by stochastic differential equations driven by fractional Brownian motion with Hurst index 
$H >\frac{1}{2}$. We give a formula type 
It\^o-Kunita-Ventzel and a substitution formula adapted to Young integral. It allows us to give necessary conditions for existence of conserved 
quantities and symmetries of Young systems. We give a formula for the composition of two flows associated to Young sistems and study the Cauchy problem for Young partial differential equations.    
\end{abstract}

\maketitle









\section{Introduction}
The aim of this article is to study differential equations driven by continuous paths with with bounded $p$-variation for $1 \leq p< 2$ from a symmetry viewpoint. In recent years there has 
been much interest in this type of equations, see  Y. Hu and D. Nualart \cite{Hu Nualart}, M. Gubinelli A. Lejay and S. Tindel \cite{Gubinelli Tindel}, A. Lejay \cite{Lejay}, X. Li and T. Lyons \cite{Li}, 
D. Nualart and R. Rascanu \cite{Nualart Rascanu}, A. Ruzmaikina \cite{Ruzmaikina}.   
This subject is a particular case of Rough Paths, when $p<2$. In fact the Theory of Rough Paths, developed by T. Lyons and colaborators (see \cite{Friz}, 
\cite{Lyons 1}, \cite{Lyons 2} and references), cover the case $p \geq 2$.

We are interested in the composition of flows of of differential equations driven by continuous paths with bounded $p$-variation (Young systems), we obtain 
a formula type It\^o-Kunita-Ventzel. We use this formula in order to obtain necessary conditions for existence of conserved quantities and symmetries. 
We study 
the composition of two flows and the Cauchy problem, 
via the characteristic method, for first order Young partial differential equations.

The plan of exposition is as follows: In section 2  we give some preliminaires on $p$-variation paths and Young integration. We prove a formula type 
It\^o-Kunita-Ventzel and a substitution formula adapted to Young integral.    

In section 3, we apply ours formulae in order to establish necessary conditions for conserved quantities and symmetries of Young systems. We give 
an adaptation of the H. Kunita result about decomposition of solutions of stochastic differential 
equations to Young systems, see  \cite{Kunita1}. Finally, we study the Cauchy problem, 
via the characteristic method, for first order Young partial differential equations of the following type
\[
du_t = \sum_{j=1}^n F^j (t, x, u_t ,D u_t )dX^j_t 
\]
where $(X^1, \cdots , X^n )$ is a path with bounded $p$-variation.

There are various different approaches to the Cauchy problem for Young partial differential equations (see \cite{Gubinelli Tindel} for a semigroup approach and 
\cite{Caruana Friz}, \cite{Deya Gubinelli} and \cite{Gubinelli Tindel1} for rough paths and SPDE). We are strongly influenced by the method of 
characteristics as developed by H. Kunita \cite{Kunita2}. In fact, we prove existence and uniqueness for Young partial differential equations under adapted 
hypoteses, following the ideas of H. Kunita.

It is clear that our results extend naturally to stochastic differential equations driven by a fractional Brownian motion with Hurst index $H >\frac{1}{2}$, 
where stochastic integrals are changed by Young integrals, see \cite{Friz}, \cite{Lyons 1} and \cite{Lyons 2}.

\section{Preliminaires}
Let $E$ and $V$ be Banach spaces. We denote by $\mathcal{P}([a,b])$ the set of all partitions
$D=\{a=t_{0}<\cdots<t_{k-1}<t_{k}=b\}$ of an interval $[a,b]$. Let
$C^{k}([0,T],E),k \in \mathbb{N}$ denote the set of $C^{k}$-class paths
of $[0,T]$ in $E$.

\begin{definition}
Let $p\in (0,\infty )$. The \emph{$p\,$-variation}
of a path $X\colon [0,T] \rightarrow E$ on the subinterval $[a,b]$ of $[0,T]$ is
defined by
\begin{equation}
\Vert X\Vert _{p,[a,b]}=(\sup_{D\in \mathcal{P}([a,b])}{\textstyle %
\sum\limits_{t_{i}\in D}}(\| X_{t_{i+1}} - X_{t_{i}})\|^{p})^{ \frac{1}{p}}\text{%
.}
\end{equation}
We say that a path $X\colon [0,T] \rightarrow E$ is of \emph{finite
$p\,$-variation} if $\Vert X\Vert _{p,[0,T]}<\infty $.
\end{definition}

If $p\in (0,1)$ and $X\colon [0,T] \rightarrow E$
is a continuous path of $p\,$-variation then $X(t)=X(0)$ for all $t\in [0,T]$,
since
\begin{equation*}
d(X_{t},X_{0})\leq {\textstyle\sum\limits_{t_{i}\in D}}
d(X_{t_{i+1}},X_{t_{i}})\leq
\max_{i}d(X_{t_{i+1}},X_{t_{i}})^{1-p}\Vert X\Vert _{p,[0,T]}^{p}\,
\end{equation*}
for all $D\in \mathcal{P}([0,T])$.

\medskip

We denote by $\mathcal{V}^{p}([0,T],E)$ the set of all continuous paths of
finite $p$-variation from $[0,T]$ to $E$. If $1\leq p\leq q<\infty $ then
\begin{equation}
\Vert X\Vert _{q,[0,T]}\leq \Vert X\Vert_{p,[0,T]}\,
\end{equation}
for each $X\colon [0,T]\rightarrow E$. In particular,
\begin{equation}
\mathcal{V}^{1}([0,T],E)\subset \mathcal{V}^{p}([0,T],E)\subset
\mathcal{V}^{q}([0,T],E)\subset C([0,T],E)\text{.}
\end{equation}

\medskip

We observe that the
set $\mathcal{V}^{p}([0,T],E)$ becomes a Banach space provided with
the norm
\begin{equation}  \label{box: p-norm}
\Vert X\Vert _{\mathcal{V}^{p}([0,T],E)}=\Vert
X\Vert_{p,[0,T]}+\sup\limits_{t\in [0,T]}\Vert X_{t}\Vert\,
\end{equation}
called \emph{$p\,$-variation norm}. We also have the \emph{$p\,$-variation
metric}
\begin{equation}
\bar{d}_{p}(X,Y)=\Vert X-Y\Vert_{\mathcal{V}^{p}([0,T],E)}\,
\end{equation}
in $\mathcal{V}^{p}([0,T],E)$ induced by the $p\,$-variation norm. We
denote by $\mathcal{V}_{0}^{p}([0,T],E)$ the subspace of
$\mathcal{V}^{p}([0,T],E)$ consisting of paths starting at $0\in E$.


\medskip

Let $E$ and $V$ be Banach spaces. Let $X\colon [0,T]\rightarrow E$ and
$Z\colon [0,T]\rightarrow\mathcal{L}(E,V)$ be continuous paths. The
Riemann-Stieltjes integral of $Z$ with respect to $X$ is defined
as the limit
\begin{equation}
\lim\limits_{\substack{|D|\rightarrow 0 \\ D\in \mathcal{P}([0,T])}}{
\textstyle\sum\limits_{s_{i}\in
D}}Z_{s_{i}}(X_{s_{i+1}}-X_{s_{i}})
\end{equation}
and is denoted by ${\textstyle\int_{0}^{t}}Z_{s}\,dX_{s}$. L. C.
Young presented the sufficient conditions for the existence of
Riemann-Stieltjes integrals. More precisely, he proved that the
integral ${\textstyle\int_{0}^{t}}Z_{s}\,dX_{s}$ exists when $X$
has finite $p\,$-variation, $Z$ has finite $q$-variation and is
valid the condition $(1/p)+(1/q)>1$. This result is known as
Young's theorem. We also have that the path $W$ given by
$W(\cdot)=\textstyle\int_{0}^{\cdot}Z_{s}\,dX_{s}$ has the same
variation of the integrator $X$, that is, $W$ has finite
$p\,$-variation. We refer the reader to the paper \cite{Young} by
L. C. Young and also \cite{Lyons 1}. Based on Young's Theorem, we
say that a Riemann-Stieltjes integral
${\textstyle\int_{0}^{t}}Z_{s}\,dX_{s}$ is an integral in the
Young sense if there exist $p,q\in[1,\infty)$ such that
$X\in\mathcal{V}^{p}([0,T],E)$,
$Z\in\mathcal{V}^{q}([0,T],\mathcal{L}(E,V))$ and
$\theta=\frac{1}{p}+\frac{1}{q}>1$. In this case holds the following Young-Loeve estimative,
\begin{equation} \label{YL}
\| \int_s^t Z_r dX_r -Z_s (X_t -X_s) \| \leq C_{p,q} \| Z \|_{q,[s,t]} \| X \|_{p,[s,t]}  
\end{equation}
where $C_{p,q}= \frac{1}{1-2^{1-\theta}}$.

We also have that   
\begin{equation}\label{independencia}
 \int_s^t Z_r dX_r = \lim\limits_{\substack{|D|\rightarrow 0 \\ D\in \mathcal{P}([s,t])}}{
\textstyle\sum\limits_{s_{i}\in
D}}Z_{s^*_{i}}(X_{s_{i+1}}-X_{s_{i}})
\end{equation}
where $s^*_i \in [s_i, s_{i+1}]$.

\begin{definition}
A path $F: [0,T] \rightarrow V$ is Holder continuous with exponent $\alpha \geq 0$, or simply $\alpha$-Holder, if
\[
 \| F \|_{\alpha; H}= \sup_{ s \neq t } \frac{\| F(x) - F(y) \|}{ |t-s |^{\alpha}} < \infty.
\]
Let $C_H^{\alpha}([0,T]; V)$ denote the set of $\alpha$-Holder paths of $V$.    
\end{definition}

We observe that $C_H^{\alpha}([0,T]; V) \subset \mathcal{V}^{\frac{1}{\alpha}}([0,T],V)$. In fact, 
\[
\| F \|_{\frac{1}{\alpha},[s,t]}^{\frac{1}{\alpha}} \leq   \| F \|_{\alpha; H}^{\frac{1}{\alpha}}|t-s|.
\]

Now, we prove a generalization of the fundamental theorem of calculus in the context of Young integration.  
\begin{lemma}\label{Ito}
 Let $X \in \mathcal{V}^{p}([0,T],V)$ and $g : [0,T] \times V \rightarrow W$ be a continuous function twice continuously differentiable in relation to $V$ (
$1 \leq p \leq 2$). 
Let $h \in C(V, C_H^{\frac{1}{q}}([0,T],L(W,U))$ and $Z \in \mathcal{V}^{p}([0,T],W)$ ($\frac{1}{p}+\frac{1}{q}> 1$) such that 
\[
 g_t (x) = g_0(x) + \int_0^t h_s(x)dZ_s
\]
where the integral is in the Young sense. Then
\begin{equation}
 g_t(X_t) = g_0(X_0) + \int_0^t h_s(X_s)dZ_s + \int_0^t D_xg_s(X_s) dX_s.
\end{equation}
\end{lemma}
\begin{proof}
Let $D=\{0=t_{0}< \cdots <t_{k-1}<t_{k}=T\} \in \mathcal{P}([0,T])$. We write, 
\begin{eqnarray}\label{II}
\sum_{t \geq t_i \in D} g_{t_{i+1}}(X_{t_{i+1}})-g_{t_i}(X_{t_i}) & = & \sum_{t \geq t_i \in D} (g_{t_{i+1}}(X_{t_{i+1}})-g_{t_{i}}(X_{t_{i+1}})) \nonumber \\
& & + \sum_{t \geq t_i \in D}(g_{t_{i}}(X_{t_{i+1}})- g_{t_i}(X_{t_i})). 
\end{eqnarray}
It follows from definitions that
\begin{eqnarray*}
 g_{t_{i+1}}(X_{t_{i+1}})-g_{t_{i}}(X_{t_{i+1}}) & = & \int_{t_i}^{t_{i+1}} h_s(X_{t_{i+1}})dZ_s \\
& = & \int_{t_i}^{t_{i+1}} (h_s(X_{t_{i+1}})-h_{t_{i}}(X_{t_{i+1}}))dZ_s + h_{t_{i}}(X_{t_{i+1}})(Z_{t_{i+1}}-Z_{t_{i}}).
\end{eqnarray*}
Taking norm and applying the Young-Loeve estimative (\ref{YL}), gives
\[
\|  \sum_{t \geq t_i \in D} \int_{t_i}^{t_{i+1}} (h_s(X_{t_{i+1}})-h_{t_i}(X_{t_{i+1}}))dZ_s \|  \leq   
\sum_{t \geq t_i \in D} C_{p,q}  \| h_{\cdot}(X_{t_{i+1}}) \|_{q,[t_i,t_{i+1}]} \| Z \|_{p,[t_i,t_{i+1}]}.
\]
Let $\theta= \frac{1}{p}+ \frac{1}{q}$ and $E(D)=\max_{t_i \in D} (\| h_{\cdot}(X_{t_{i+1}}) \|_{q,[t_i,t_{i+1}]} 
\| Z \|_{p,[t_i,t_{i+1}]})^{1 -\frac{1}{\theta}}$. Applying the Holder inequality ($\frac{1}{p \theta}+\frac{1}{q \theta}=1$) and some elementary 
calculations, we have that  
\begin{eqnarray*}
\sum_{t \geq t_i \in D}\| h_{\cdot}(X_{t_{i+1}}) \|_{q,[t_i,t_{i+1}]} \| Z \|_{p,[t_i,t_{i+1}]} & \leq & E(D) (\sum_{t \geq t_i \in D}
(\| h_{\cdot}(X_{t_{i+1}}) \|_{q,[t_i,t_{i+1}]} 
\| Z \|_{p,[t_i,t_{i+1}]})^{\frac{1}{\theta}} \\ 
& \leq &  E(D) (\sum_{t \geq t_i \in D}\| h_{\cdot}(X_{t_{i+1}}) \|^q_{q,[t_i,t_{i+1}]})^{\frac{1}{q \theta}}  
 (\sum_{t \geq t_i \in D} \| Z \|^p_{p,[t_i,t_{i+1}]})^{\frac{1}{p\theta}} \\          
& \leq & E(D)  \| Z \|^{\frac{1}{\theta}}_{p,[0,T]}(\sum_{t \geq t_i \in D}\| h_{\cdot}(X_{t_{i+1}}) \|^q_{q,[t_i,t_{i+1}]})^{\frac{1}{q \theta}} \\
& \leq & E(D)  \| Z \|^{\frac{1}{\theta}}_{p,[0,T]}(\sum_{t \geq t_i \in D}\| h_{\cdot}(X_{t_{i+1}}) \|^q_{\frac{1}{q},H}(t_{i+1}-t_i) 
)^{\frac{1}{q \theta}} \\
& \leq &  E(D)  \| Z \|^{\frac{1}{\theta}}_{p,[0,T]}(\sup_{s}\| h_{\cdot}(X_s) \|^q_{\frac{1}{q},H} T)^{\frac{1}{q \theta}}.
\end{eqnarray*}
Taking limit and using (\ref{independencia}),  
\begin{equation}\label{aa}
\lim_{|D| \rightarrow 0} \sum_{t \geq t_i \in D} (g_{t_{i+1}}(X_{t_{i+1}})-g_{t_{i}}(X_{t_{i+1}}))= \int_0^t h_s(X_s)dZ_s, 
\end{equation}
because $\lim_{|D| \rightarrow 0}E(D)=0$.

We claim that 
\begin{equation}\label{11}
\lim_{|D| \rightarrow 0} \sum_{t \geq t_i \in D}(g_{t_{i}}(X_{t_{i+1}})- g_{t_i}(X_{t_i}))=\int_0^t D_xg_s(X_s) dX_s.
\end{equation}
In fact, by Taylor's theorem
\begin{eqnarray*}
 g_{t_{i}}(X_{t_{i+1}})- g_{t_i}(X_{t_i}) & = & D_xg_{t_i}(X_{t_i}) \cdot (X_{t_{i+1}}-X_{t_i}) +\frac{1}{2} D^2_xg_{t_i}(X_{t_i}+s_i(X_{t_{i+1}}-X_{t_i})) 
 \\ 
&  & \cdot  
(X_{t_{i+1}}-X_{t_i})^2
\end{eqnarray*}
where $s_i \in (0,1)$.
Taking norm,
\begin{eqnarray*}
\| \sum_{t \geq t_i \in D} D^2_xg_{t_i} (X_{t_i}+s_i(X_{t_{i+1}}-X_{t_i})) \cdot  
(X_{t_{i+1}}-X_{t_i})^2  \| & \leq &  K \sum_{t \geq t_i \in D}  \| X_{t_{i+1}}-X_{t_i} \|^2
\end{eqnarray*}
where $K=\max \{ \| D^2_xg_{s}(X_{s}+r(X_{t}-X_{s})) \| : 0 \leq s \leq t \leq T~\mbox{ and } r \in [0,1] \} $. 
Since $X\in \mathcal{V}^{p}([0,T],V)$ with $1 \leq p < 2$ it follows that $Dg(X) \in \mathcal{V}^{p}([0,T],L(V,W))$. Combining the above estimative 
and definitions we have (\ref{11}).

Finally, from the continuity of $g$, 
\begin{equation}\label{4}
g_t(X_t)-g_0(X_0) = \lim_{|D| \rightarrow 0} \sum_{t \geq t_i \in D}  g_{t_{i+1}}(X_{t_{i+1}})-g_{t_i}(X_{t_i}).
\end{equation}
Taking limit in  (\ref{II}) and then substituing (\ref{aa}), (\ref{11}) and (\ref{4}), we obtain 
\[
g_t(X_t) = g_0(X_0) + \int_0^t h_s(X_s)dZ_s + \int_0^t D_xg_s(X_s) dX_s.
\]

\end{proof}
  
\begin{corollary}
 Let $g : V \rightarrow W$ be a twice differentiable function and $Z\in \mathcal{V}^{p}([0,T],V)$ ($1 \leq p < 2$). Then 
$g(Z) \in \mathcal{V}^{p}([0,T],W)$, $Dg(Z) \in \mathcal{V}^{p}([0,T],L(V,W))$ and for all $0 \leq t \leq T$, 
\[
g(Z_t) - g(Z_0) = \int_0^t Dg(Z_r)dZ_r.
\]
In particular, 
\begin{equation}
 dg(Z_t)=Dg(Z_t)dZ_t.
\end{equation}
\end{corollary}
The following substitution formula holds.
\begin{lemma}\label{Sub}
Let $Z\in \mathcal{V}^{p}([0,T],V)$, $f \in \mathcal{V}^{q}([0,T],Hom(V,W))$ and 
$g \in \mathcal{V}^{l}([0,T],Hom(W,U))$ where $\frac{1}{q}, \, \frac{1}{l}>1-\frac{1}{p}$.
Then for all $0 \leq s \leq t \leq T$,
\begin{equation}
\int_s^t g_r dY_r = \int_s^t g_r \circ f_rdZ_r 
\end{equation}

where $Y_t =\int_0^t f(Z_r)dZ_r$.
  
\end{lemma}
\begin{proof}
Let $D=\{0=t_{0}<...<t_{k-1}<t_{k}=T\} \in \mathcal{P}([0,T])$. Then
\[
\sum_{t \geq t_i \in D} g_{t_i}(Y_{t_{i+1}}-Y_{t_i}) - \sum_{t \geq t_i \in D} g_{t_i} f_{t_i}(Z_{t_{i+1}}-Z_{t_i}) = 
\sum_{t \geq t_i \in D} g_{t_i}\int_{t_i}^{t_{i+1}}(f_r -f_{t_i})dZ_r.
\]
Taking norm and applying the Young-Loeve estimative (\ref{YL}), gives
\begin{eqnarray*}
\|  \sum_{t \geq t_i \in D} g_{t_i}\int_{t_i}^{t_{i+1}}(f_r -f_{t_i})dZ_r \| & \leq & \sum_{t \geq t_i \in D} \|g_{t_i}\| \| \int_{t_i}^{t_{i+1}}
(f_r -f_{t_i})dZ_r \| \\ 
& \leq & \sum_{t \geq t_i \in D} K C_{p,q}  \| f \|_{q,[t_i,t_{i+1}]} \| Z \|_{p,[t_i,t_{i+1}]} .
\end{eqnarray*}
Let $\theta= \frac{1}{p}+ \frac{1}{q}$ and $L(D)=\max_{t_i \in D} (\| f \|_{q,[t_i,t_{i+1}]} 
\| Z \|_{p,[t_i,t_{i+1}]})^{1 -\frac{1}{\theta}}$. Applying the Holder inequality ($\frac{1}{p \theta}+\frac{1}{q \theta}=1$), we have that  
\begin{eqnarray*}
\sum_{t \geq t_i \in D}\| f \|_{q,[t_i,t_{i+1}]} \| Z \|_{p,[t_i,t_{i+1}]} & \leq & L(D) (\sum_{t \geq t_i \in D}(\| f \|_{q,[t_i,t_{i+1}]} 
\| Z \|_{p,[t_i,t_{i+1}]})^{\frac{1}{\theta}} \\ 
& \leq &  L(D) (\sum_{t \geq t_i \in D}\| f \|^q_{q,[t_i,t_{i+1}]})^{\frac{1}{q \theta}} \cdot 
 (\sum_{t \geq t_i \in D} \| Z \|^p_{p,[t_i,t_{i+1}]})^{\frac{1}{p\theta}} \\          
& \leq & L(D) \| f \|^{\frac{1}{\theta}}_{q,[0,T]}) \| Z \|^{\frac{1}{\theta}}_{p,[0,T]}.
\end{eqnarray*}

\noindent Combining the three above inequalities we obtain that 
\begin{equation}\label{a1}
 \|\sum_{t \geq t_i \in D} g_{t_i}(Y_{t_{i+1}}-Y_{t_i}) - \sum_{t \geq t_i \in D} g_{t_i} f_{t_i}(Z_{t_{i+1}}-Z_{t_i})\| \leq \tilde{K}L(D).
\end{equation}
From the continuity of $\| f \|_{q,[s,t]}$ $\| Z \|_{p,[s,t]}$ we have that 
\begin{equation}\label{a2}
\lim_{|D| \rightarrow 0} L(D)=0. 
\end{equation}
Combining (\ref{a1}) and (\ref{a2}) we have that
\[
\lim_{|D| \rightarrow 0} \|\sum_{t \geq t_i \in D} g_{t_i}(Y_{t_{i+1}}-Y_{t_i}) - \sum_{t \geq t_i \in D} g_{t_i} f_{t_i}(Z_{t_{i+1}}-Z_{t_i}) \| =0.  
\]
\end{proof}

\section{Young systems}

Let\ $E_{1}$ and
$E_{2}$ be Banach spaces, $p\in[1,2 )$ and $f$ be a function from $E_1$ to $L (E_2 , E_1 )$ which is $\gamma$-Holder continuous,
 $\gamma \in (0,1]$. We call such a function a $Lip(\gamma)$-vector field from $E_1$ to $E_2$.
For $Y \in \mathcal{V}^{p}([0,T],E_{1})$ we have that  $f ( Y )$ belongs to $\mathcal{V}^{\frac{p}{\gamma}}([0,T],L(E_{2},E_1))$ and
\[
\|f(Y)\|_{\frac{p}{\gamma}} \leq \| f \|_{\gamma , H} \| Y \|_p^{\gamma}. 
\]
We suppose that $\gamma +1 >p$. From Young's theorem it is clear that for $X \in \mathcal{V}^{p}([0,T],E_{2})$ there exists 
the Young integral
\[
 \int_0^t f(Y_s)dX_s.
\]

Let $X \in \mathcal{V}^{p}([0,T],E_{2})$ and $f$ be a  $Lip(\gamma)$-vector field from $E_1$ to $E_2$ with
$\gamma\in(0,1]$. Given an initial
condition $y_{0}\in E_{1}$, we understand that a \emph{trajectory}
is a path $Y$ of finite $p$-variation in $E_{1}$ which is the
solution starting at $y_{0}$ of an equation of type
\begin{equation}  \label{box: reduced equation of Young}
dY=f(Y)\,dX
\end{equation}
in the sense that $Y(t)=y_{0}+{\textstyle\int_{0}^{t}}f(Y
(s))\,dX_{s}$, for all $t\in [0,T]$.

\medskip

We will call the equation (\ref{box: reduced equation of Young})
by Young equation. The Young equation $dY=f(Y)\,dX$ admits 
solution starting at $y_{0}\in E_{2}$. 
In order of obtain uniqueness we need a stronger regularity assumption on $f$, for example we can assume
that $f$ is a $Lip(1 + \gamma )$-vector field, this is, $f$ is continuously differentiable and its derivative is a $\gamma$-Holder 
continuous function from from $E_1$ to $L (E_1 \otimes E_2 , E_1 )$.
Moreover, if
$I_{f}(y_{0},X)$ denotes a solution starting at $y_{0}$ then the
mapping $(y_{0},X)\mapsto I_{f}(y_{0},X)$ is an diffeomorphism. 

In the case that $E_1$ and $E_2$ are finite dimensional spaces, we have similar results on existence and uniquenness of solutions for the  
Young equation driven by time dependent fields 

\begin{equation}\label{young finite}
dY_s = \sum_{i=1}^n f_i (s, Y_s) \, dX_s^i
\end{equation}
where $f_i : [0,T] \times E_2 \rightarrow E_2$ are the vector fields $\gamma$-Holder with $1+\gamma >1$ in space and uniformly of finite $q$-variation in time with 
$\frac{1}{p}+ \frac{1}{q}>1$, see \cite{Lejay}.

We refer the reader to \cite{Friz}, \cite{Lejay}, \cite{Li}, \cite{Lyons 1} and \cite{Lyons 2} for more information about
existence and uniquennes solutions of Young equations.

\subsection{Symmetries and invariants}

The theory of conserved quantities (first integrals) and symmetry (invariant under transformation) for dynamical systems must been one 
of the most important subjets in applied mathematics, see \cite{Bluman Anco}, \cite{Grigoriev} and \cite{Olver}. Hence, it is natural to 
formulate theses notions for Young systems. In this subsection we consider only homogeneous Young systems. 
\begin{definition}
A function $F \in C^2(E_1; E_3)$ is a
conserved quantity of (\ref{box: reduced equation of Young}) if for each solution $Y$ of (\ref{box: reduced equation of Young}) we have that 
$F(Y(t))=F(Y(0))$ for all $t \in [0,T]$.
\end{definition}
The following necessary condition for a function be a conserved quantity of (\ref{box: reduced equation of Young}) is an immediate consequence of 
Lemmas \ref{Ito} and \ref{Sub}.
\begin{corollary}
Let $F \in C^2(E_1; E_3)$ such that 
 $DF \cdot f=0$.
Then $F$ is a conserved quantity of (\ref{box: reduced equation of Young}). 
\end{corollary}
We formulate the notion of symmetry for Young systems in an analogous way to that in differentiable dynamical systems. We are interest 
in Lie point time independent symmetries.
\begin{definition}
A transformation $\Phi \in C^2(E_1; E_1)$ is a
simmetry of (\ref{box: reduced equation of Young}) if for each solution $Y$ of (\ref{box: reduced equation of Young}) we have that 
$\Phi (Y)$ is also a solution of (\ref{box: reduced equation of Young}).
\end{definition}

\begin{proposition}
Let $\Phi \in C^1(E_1 ; E_1)$ such that 
 $f \circ \Phi = D\Phi \cdot f$.
Then $\Phi$ is a symmetry of (\ref{box: reduced equation of Young}). 
\end{proposition}
\begin{proof}
Applying the Lemma \ref{Ito} and Proposition \ref{Sub} we have that,  
\begin{eqnarray*}
\Phi(Y_t)-\Phi(Y_0) & = & \int_0^t D\Phi(Y_s) f(Y_s)dX_s \\
& = & \int_0^t f(\Phi(Y_s)) dX_s.
\end{eqnarray*}
\end{proof}
Let $\{e_1, \cdots ,e_n \}$ be a basis of $E_2$. Thus $X_s = \sum_{i=1}^n X^i_s e_i$ and we can write the Young equation 
(\ref{box: reduced equation of Young}) as
\begin{equation}\label{young finite}
dY_s = \sum_{i=1}^n f_i (Y_s) \, dX_s^i
\end{equation}
where $f_i : E_1 \rightarrow E_1$ are the vector fields given by $f_i(y)=f(y)(e_i)$ for $i=1, \cdots , n$.

\begin{proposition}\label{cq}
Let  $Y$ be a solution of (\ref{young finite}). Then for all $F \in C^2(E_1; E_3)$, 
\begin{equation}\label{eq}
 F(Y_t)= F(Y_0) +\sum_{i=1}^n \int_0^t  f_iF(Y_s)dX^i_s.
\end{equation} 
\end{proposition}
\begin{proof}
Applying the Lemma \ref{Ito} and Proposition \ref{Sub} we have that,  
\begin{eqnarray*}
 F(Y_t)-F(Y_0) & = & \int_0^t DF(Y_s) f(Y_s)dX_s \\
& = & \int_0^t DF(Y_s) \sum_{i=1}^n f(Y_s)(e_i) dX^i_s \\
& = & \sum_{i=1}^n  \int_0^t DF(Y_s) f_i(Y_s) dX^i_s \\
& = & \sum_{i=1}^n  \int_0^t f_iF(Y_s) dX^i_s.
\end{eqnarray*}
\end{proof}
\begin{definition}
A vector field $g \in C^2(E_1; E_1)$ is an infinitesimal symmetry of (\ref{box: reduced equation of Young}) if its flow $\Phi_t$ is a flow of symmetries of 
(\ref{box: reduced equation of Young}).
\end{definition}

The following Corollaries provide conditions for a transformation be a conserved quantity or a symmetry in terms of the 
vector fields that are driven the Young equation.
\begin{corollary}
Let $F \in C^1(E_1; E_3)$ such that 
 $f_i F = 0$ for $i=1,\cdots, n$.
Then $F$ is a conserved quantity of (\ref{young finite}). 
\end{corollary}

\begin{corollary}
Let $\Phi \in C^1(E_1; E_1)$ such that 
 $\Phi_* f_i  = f_i$ for $i=1,\cdots, n$.
Then $\Phi$ is a symmetry of (\ref{young finite}). 
\end{corollary}

\begin{proposition}
Let $g \in C^2(E_1; E_1)$ be a vector field such that $[g, f_i]=0$ for $i=1,\cdots ,n$. Then $g$ is an infinitesimal symmetry of (\ref{young finite}).   
\end{proposition}
\begin{proof}
Let $\Phi_t$ be the flow of $g$. Then
\[
 \partial_t \Phi^*_t f_i = \partial_s |_{s=0} \Phi^*_{t+s} f_i = \Phi^*_t [g,f_i]=0 
\]
so $\Phi^*_t f_i$ is constant in $t$. Thus $f_i =\Phi^*_t f_i$.
\end{proof}

The following Theorem is an adaptation of the H. Kunita results about decomposition of solutions of stochastic differential 
equations to Young systems, see  \cite{Kunita1}.

\begin{theorem}
Let $p\in[1,2)$, $p<\gamma$, 
$U \in \mathcal{V}^{p}([0,T],E_{0})$, $X \in \mathcal{V}^{p}([0,T],E_{1})$, $f\in
\mathrm{Lip}^{\gamma}(E,L(E_{0},E))$ and $g\in
\mathrm{Lip}^{\gamma}(E,L(E_{1},E))$. Let $V$ and $Y$ be solutions of $dV = f(V)dU$ and $dY = g(Y) dX$. Then $Z= Y \circ V$ satisfies
\[
 dZ = g(Z)dX + Y_*f(Z)dU.
\]
\end{theorem}
\begin{proof}
By assumption,
\begin{equation}
Y_t(x) = x + \int_0^t g(Y_s(x))dX_s.  
\end{equation}
Combining Lemma \ref{Ito} and Proposition \ref{Sub} we have that 
\begin{eqnarray*}
Z_t & = & Y_t(V_t)  \\
& = & x + \int_0^t g(Y_s(V_s))dX_s + \int_0^t D_xY_s(V_s)dV_s \\
& = & x + \int_0^tg(Z_s)dX_s + \int_0^t D_xY_s(V_s)f(V_s)dU_s \\ 
& = & x + \int_0^tg(Z_s)dX_s + \int_0^t (D_xY_s\cdot f ) \circ Y_s^{-1} (Z_s)dU_s \\
& = & x + \int_0^tg(Z_s)dX_s + \int_0^t (Y_s)_* f(Z_s)dU_s
\end{eqnarray*}

\end{proof}

\subsection{First order Young partial differential equations}

In this section we deal with a class of evolution first order differential equations driven by a path with finite $p$-variation, with  $p\in[1,2 )$. 
Let $X \in \mathcal{V}^{p}([0,T],\mathbb{R}^n)$, $\phi : \mathbb{R}^d \rightarrow \mathbb{R}$ and $F^j : [0,T] \times \mathbb{R}^d \times \mathbb{R} \times \mathbb{R}^d \rightarrow \mathbb{R}$ 
, $j=1, \cdots , n$. We will consider the following equation

\begin{equation}\label{transport}
 \left \{
\begin{array}{lll}
    du_t = \sum_{j=1}^n F^j (t, x, u_t ,D u_t )dX^j_t\\
 u_0 = \phi.
\end{array}
\right .
\end{equation}
We assume that  $F^j$ are continuous, continuously differentiable in the first variavel and for each $t$, 
$F^j (t, \cdot) \in \mathcal{C}^{3,\alpha}(\mathbb{R}^{2d+1})$.
  
\begin{definition}
Given $\phi \in \mathcal{C}^1(\mathbb{R}^d)$ a local field $u_t(x)$, $x \in \mathbb{R}^d$ $t \in [0, T(x))$ with values in $\mathbb{R}$ 
is called a local solution of (\ref{transport}) with the initial condition $u_0 = \phi$, if $0< T(x) \leq T$ and 
\[
 u(t,x) = \phi(x) + \sum_{j=1}^n \int_0^t F^j (r , x , u(r,x), D_x u(r,x)) dX^j_r
\]
for all $(t,x)$ such that $t< T(x)$.  
\end{definition}
We use the following notations $F_{x_i}= D_{x_i}F$, $F_{p_i}= D_{p_i}F$,  $F_x =(F_{x_1}, \cdots , F_{x_d})$ and $F_p =(F_{p_1}, \cdots , F_{p_d})$.

The characteristic Young system associated with (\ref{transport}) is defined by
\begin{equation}\label{characteristic}
 \left \{
\begin{array}{lll}
    da_t & = & -\sum_{j=1}^n F_p^j (t, a_t, b_t ,c_t )dX^j_t\\
 db_t & = & \sum_{j=1}^n \{ F^j (t, a_t, b_t ,c_t )-F_p^j (t, a_t, b_t ,c_t ) \cdot c_t \} dX^j_t \\
dc_t & = & \sum_{j=1}^n \{ F_x^j (t, a_t, b_t ,c_t ) + F_u^j (t, a_t, b_t ,c_t ) b_t \} dX^j_t
\end{array}
\right .
\end{equation}
Given $(x,u,p) \in \mathbb{R}^d \times \mathbb{R} \times \mathbb{R}^d$ there is an unique solution $(a_t (x,u,p), b_t (x,u,p) , c_t (x,u,p))$ 
starting from $(x,u,p)$ at time $t=0$ with time life $[0,T(x,u,p))$.
\begin{theorem}\label{T1}
Let $u$ be a local solution of (\ref{transport}) such that $u(t, \cdot) \in C^3(\mathbb{R}^d)$  for all $t \in [0,T]$. Assume that $a_t$ solves the equation
\begin{equation}\label{E1}
 da_t  =  -\sum_{j=1}^n F_p^j (t, a_t, b_t ,c_t )dX^j_t
\end{equation}
where $b_t = u(t, a_t)$ and $c_t = D_xu (t, a_t)$. Then $(a_t, b_t , c_t)$ solves the characteristic system (\ref{characteristic}).   
\end{theorem}
 \begin{proof}
By Lemma \ref{Ito},
\begin{eqnarray*}
 u(t,a_t) & = & \phi(a_0) + \sum_{j=1}^n \int_0^t F^j (r , a_r , u(r,a_r), D_x u(r,a_r)) dX^j_r \\
& & + \int_0^t D_x u(r,a_r)da_r.
\end{eqnarray*}
Thus
\begin{equation}\label{E0} 
b_t = \phi(a_0) + \sum_{j=1}^n \int_0^t F^j (r , a_r , b_r, c_r) dX^j_r + \int_0^t D_x u(r,a_r)da_r.
\end{equation}
From Lemma \ref{Sub} and (\ref{E1}), we have 
\begin{eqnarray}\label{E2}
\int_0^t D_x u(r,a_r)da_r & = & -\sum_{j=1}^n \int_0^t D_x u(r,a_r) F_p^j (r, a_r, b_r ,c_r )dX^j_r \nonumber \\
 & = & -\sum_{j=1}^n \int_0^t F_p^j (r, a_r, b_r ,c_r ) \cdot c_t dX^j_r.
\end{eqnarray}
Combining (\ref{E0}) with (\ref{E2}) we obtain
\begin{equation}\label{E11}
 b_t = \phi(a_0) + \sum_{j=1}^n \int_0^t \{ F^j (r , a_r , b_r, c_r) - F_p^j (r, a_r, b_r ,c_r ) \cdot c_t \}dX^j_r .
\end{equation}
Our next goal is to determine the equation for $c_t$. We observe that
\begin{eqnarray*}
u_{x_i}(t,x) & = & \phi_{x_i}(x)+ \sum_{j=1}^n \int_0^t D_{x_i}F^j (r , x , u(r,x), D_x u(r,x)) dX^j_r \\
  & = & \phi_{x_i}(x)+ \sum_{j=1}^n \int_0^t \{ F_{x_i}^j (r , x , u(r,x), D_x u(r,x)) \\
& & + F_{u}^j (r , x , u(r,x), D_x u(r,x))u_{x_i}(r, x) \\
   & & + \sum_{l=1}^d F^j_{p_l}(r,x,u(r,x), D_xu(r,x))u_{x_l x_i}(r,x) \} dX^j_r.
\end{eqnarray*}
By Lemma \ref{Ito} and definitions, 
\begin{eqnarray}\label{E4}
 c^i_t & = & \phi_{x_i}(a_0) + \sum_{j=1}^n \int_0^t \{ F_{x_i}^j (r ,a_r , b_r, c_r) + F_{u}^j(r ,a_r , b_r, c_r)c^i_r \nonumber \\
 & & + \sum_{l=1}^d F^j_{p_l}(r ,a_r , b_r, c_r) u_{x_l x_i}(r,a_r) \} dX^j_r + \int_0^t D_x u_{x_i}(r,a_r)da_r
\end{eqnarray}
 where $c^i_t= u_{x_i}(t, a_t)$. 
 
From Lemma \ref{Sub} and (\ref{E1}), we have 
\begin{eqnarray}\label{E5}
\int_0^t D_x u_{x_i}(r,a_r)da_r & = & -\sum_{j=1}^n \int_0^t D_x u_{x_i}(r,a_r) F_p^j (r, a_r, b_r ,c_r )dX^j_r \nonumber \\
 & & -\sum_{j=1}^n \int_0^t \sum_{l=1}^d F^j_{p_l}(r ,a_r , b_r, c_r) u_{x_l x_i}(r,a_r) dX^j_r.
\end{eqnarray}
 
Combining (\ref{E4}) with (\ref{E5}) we obtain
\begin{equation}\label{E22}
  c^i_t  =  \phi_{x_i}(a_0) + \sum_{j=1}^n \int_0^t \{ F_{x_i}^j (r ,a_r , b_r, c_r) + F_{u}^j(r ,a_r , b_r, c_r)c^i_r \}dX^j_r .
\end{equation}
This is
\begin{equation}\label{E22}
  c_t  =  \phi_{x_i}(a_0) + \sum_{j=1}^n \int_0^t \{ F_{x}^j (r ,a_r , b_r, c_r) + F_{u}^j(r ,a_r , b_r, c_r)c_r \}dX^j_r .
\end{equation}

 \end{proof}

Following H. Kunita \cite{Kunita2} we define $\overline{a}_t(x)= a_t(x, \phi(x), D\phi(x))$, $\overline{b}_t(x))= b_t(x, \phi(x), D\phi(x))$ and 
$\overline{c}_t(x) = c_t(x, \phi(x), D\phi(x))$ for $t \in [0,\overline{T}(x)$ where $\overline{T}(x)=T(x, \phi(x), D\phi(x))$.
We observe that $\overline{a}_t : \{ x : \overline{T}(x) >t \} \rightarrow \mathbb{R}^d$ is not a diffeomorphism in general, since $D \overline{a}_t(x)$ 
can be singular at some $t< \overline{T}(x)$. 
We define 
\begin{equation}
 \tau(x) = \inf \{ t>0 : \det D \overline{a}_t(x) =0 \} \wedge \overline{T}(x)
\end{equation}
and  its adjoint is given by   
\begin{equation}
 \sigma(y) = \inf \{ t>0 : y \notin \overline{a}_t(\{ \tau > t \})  \}.
\end{equation}
The proofs of the following two Lemmas are an easy adaptation of Lemma 2.1 and Lemma 3.3 of \cite{Kunita2}. 
\begin{lemma}\label{L1}
 The application $\overline{a}_t : \{ x : \overline{T}(x) >t \} \rightarrow \mathbb{R}^d$  is a diffeomorphism. For $t < \sigma(y)$, the inverse 
$\overline{a}^{-1}_t$ satisfies 
\begin{equation}
 d \overline{a}^{-1}_t(y) = \sum_{j=1}^n D\overline{a}_t(\overline{a}^{-1}_t(y))^{-1} F^j_p(t, y , 
\overline{b}_t \circ\overline{a}^{-1}_t(y) , \overline{c}_t \circ\overline{a}^{-1}_t(y))dX^j_t.
\end{equation}
\end{lemma}{\label{L2}}
\begin{lemma} \label{formula}
For $i = 1,\cdots , d$
 \begin{eqnarray*}
  D_{x_i}\overline{b}_t & = & \overline{c}_t \cdot D_{x_i} \overline{a}_t \\
D_{x_i}\overline{b}_t  \circ \overline{a}_t^{-1} & = & \overline{c}^i_t \circ \overline{a}^{-1}.
 \end{eqnarray*}
\end{lemma}

\begin{theorem}
Let $\phi \in C^3(\mathbb{R}^d)$. Then  $u(t,x)= \overline{b}_t(\overline{a}_t^{-1}(x))$, $[0,\sigma(x))$ is a local solution of (\ref{transport}).
\end{theorem}
\begin{proof}
By Lemma \ref{Ito},
\begin{equation}\label{eq0}
 d\overline{b}_t \circ \overline{a}_t^{-1} = d \overline{b}_t(\overline{a}_t^{-1}) + D_x \overline{b}_t (\overline{a}_t^{-1}) d\overline{a}_t^{-1}.  
\end{equation}
From (\ref{characteristic}) we have 
\begin{eqnarray}\label{eq1}
 d \overline{b}_t(\overline{a}_t^{-1}) & = & \sum _{j=1}^n\{ F^j(t, \cdot, \overline{b}_t \circ \overline{a}_t^{-1}, \overline{c}_t 
\circ \overline{a}_t^{-1}) \\ 
& & -F^j_p(t, \cdot, \overline{b}_t \circ \overline{a}_t^{-1}, \overline{c}_t \nonumber
\circ \overline{a}_t^{-1}) \overline{c}_t 
\circ \overline{a}_t^{-1} \} dX^j_t
\end{eqnarray}
By Lemma \ref{Sub} and Lemma \ref{L1},
\begin{eqnarray}\label{eq2}
 D_x \overline{b}_t (\overline{a}_t^{-1}) d\overline{a}_t^{-1}  & = & \sum _{j=1}^n F_p^j(t, \cdot, \overline{b}_t \circ \overline{a}_t^{-1}, 
\overline{c}_t 
\circ \overline{a}_t^{-1})  \\ 
& & D_x\overline{b}_t(\overline{a}_t^{-1})D_x\overline{a}_t(\overline{a}^{-1}_t)^{-1}dX^j_t. \nonumber
\end{eqnarray}
From Lemma \ref{L2} and definitions,
\begin{equation}\label{eq3}
D_xu_t= 
D_x(\overline{b}_t \circ \overline{a}^{-1}_t) = \overline{c}_t \circ \overline{a}_t^{-1}= 
D_x\overline{b}_t(\overline{a}_t^{-1})D_x\overline{a}_t(\overline{a}^{-1}_t)^{-1}.
\end{equation}

Combining (\ref{eq0}), (\ref{eq1}), (\ref{eq2}) and (\ref{eq3}) we conclude that
\[
 du_t = \sum _{j=1}^n F^j(t, \cdot, \overline{b}_t \circ \overline{a}_t^{-1}, \overline{c}_t 
\circ \overline{a}_t^{-1})dX^j_t = \sum _{j=1}^n F^j(t, \cdot,u_t, D_xu_t) dX^j_t.
\]
\end{proof}
\begin{theorem}
 Let $u$ be a local solution of (\ref{transport}), where $\phi \in C^3(\mathbb{R}^d)$ such that $u(t, \cdot) \in C^2(\mathbb{R}^d)$ 
for all $t \in [0,T(x))$. Then $u(t,x)=\overline{b}_t(\overline{a}_t^{-1}(x))$ for $t \in [0,T(x) \wedge \sigma(x))$.
\end{theorem}
\begin{proof}
 It is an easy consequence of Theorem \ref{T1}.
\end{proof}


\end{document}